\documentclass[12pt]{amsart}

\usepackage[utf8]{inputenc}
\usepackage{microtype}
\usepackage{amsfonts}
\usepackage{amsmath}
\usepackage{graphicx}
\usepackage{float}
\usepackage{color}
\usepackage{hyperref}
\usepackage{tikz-cd}

\usepackage{amssymb}
\usepackage{enumitem}
\usepackage{mathrsfs}

\usepackage{tikz}
\usetikzlibrary{topaths}
\usetikzlibrary{calc}

\usepackage{a4wide}

\usepackage{empheq}

\newtheorem{theorem}{Theorem}[section]
\newtheorem*{theorem*}{Theorem}

\newtheorem*{proposition*}{Proposition}
\newtheorem{corollary}[theorem]{Corollary}
\newtheorem*{corollary*}{Corollary}
\newtheorem{conjecture}[theorem]{Conjecture}

\newtheorem*{conjecture*}{Conjecture}
\newtheorem{question}[theorem]{Question}
\newtheorem*{question*}{Question}

\newtheorem*{main:main}{Theorem~\ref{thrm:good_action_fp}}
\newtheorem*{main:sep_all_raags}{Corollary~\ref{cor:sep_all_raags}}

\theoremstyle{definition}
\newtheorem{definition}[theorem]{Definition}

\newtheorem{example}[theorem]{Example}

\newcommand{\Z}{\mathbb{Z}}
\newcommand{\N}{\mathbb{N}}

\newcommand{\C}{\mathfrak{C}}

\newcommand{\defeq}{\mathbin{\vcentcolon =}}

\DeclareMathOperator{\Aut}{Aut}
\DeclareMathOperator{\Out}{Out}
\DeclareMathOperator{\AAut}{AAut}

\DeclareMathOperator{\Symm}{Symm}

\DeclareMathOperator{\F}{F}

\newcommand{\brF}%
   {\operatorname{br}\!F}                 

\newcommand{\brV}%
   {\operatorname{br}\!V}                 
   
\newcommand{\brAut}%
   {\operatorname{br}\!\Aut}                 
   
\newcommand{\brAAut}%
   {\operatorname{br}\!\AAut}                 
   
\newcommand{\brW}%
   {\operatorname{br}\!W}                 
   
\newcommand{\brR}%
   {\operatorname{br}\!R}
   
\numberwithin{equation}{section}

\begin{document}

\title{A taste of twisted Brin--Thompson groups}
\date{\today}


\author[M.~C.~B.~Zaremsky]{Matthew C.~B.~Zaremsky}
\address{Department of Mathematics and Statistics, University at Albany (SUNY), Albany, NY 12222}
\email{mzaremsky@albany.edu}

\begin{abstract}
This note serves as a short and reader-friendly introduction to twisted Brin--Thompson groups, which were recently constructed by Belk and the author to provide a family of simple groups with a variety of interesting properties. Most notably, twisted Brin--Thompson groups can be used to show that every finitely generated group quasi-isometrically embeds as a subgroup of a finitely generated simple group. Another important application is a concrete construction of a family of simple groups with arbitrary finiteness length. In addition to giving a concise introduction to the groups and these applications, we also prove here a strengthening of one of the results from the original paper. Namely, we prove that any finitely presented group acting faithfully and oligomorphically on a set, with finitely generated stabilizers of finite subsets, embeds in a finitely presented simple group. We believe this could potentially lead to future progress on resolving the Boone--Higman Conjecture for certain groups.
\end{abstract}

\maketitle
\thispagestyle{empty}

\section{Introduction}

The Brin--Thompson groups $sV$ ($s\in\N$), also called higher-dimensional Thompson groups, were originally introduced by Brin in \cite{brin04}. They serve as a nice, natural infinite family of finitely presented infinite simple groups, generalizing the classical Thompson's group $V=1V$. The fact that they are all finitely generated and simple was proved by Brin \cite{brin04,brin10}, and finite presentability was shown by Hennig--Matucci \cite{hennig12}. Since their introduction, these groups have turned out to have an array of remarkable properties. For instance, while $V=1V$ has decidable torsion problem, Belk--Bleak showed that none of the other $sV$ do \cite{belk17}. Also, Belk--Bleak--Matucci proved that every right-angled Artin group (RAAG) embeds into some $sV$ \cite{belk20}, and very recently Salo showed that every RAAG already embeds into $2V$ \cite{salo}. This is striking, since very few RAAGs embed into $V$, by virtue of $V$ not containing $\Z*\Z^2$ \cite{bleak13}. Another important result, due to Matte Bon, is that $mV$ embeds into $nV$ if and only if $m\le n$ \cite{mattebon}. Work has also been done to understand the metric properties of $sV$ \cite{burillo10} and higher finiteness properties \cite{kochloukova13,fluch13}; in particular it turns out every $sV$ is of type $\F_\infty$.

In \cite{belkZ}, Jim Belk and the author introduced ``twisted'' variants of the Brin--Thompson groups. Given $G\le \Symm(\{1,\dots,s\})$, the version of $sV$ twisted by $G$ is denoted $sV_G$. More generally, given any set $S$ and any $G\le \Symm(S)$, one gets a twisted Brin--Thompson group $SV_G$. The idea is, $sV$ acts on the direct product $\C^s$ of $s$ many copies of the Cantor set $\C$, and in $sV_G$ one allows $G$ to shuffle these copies around, potentially in different ways in different regions of $\C^s$. When using an infinite set $S$, the Brin--Thompson group $SV$ acts on the product space $\C^S$, and the twisted version $SV_G$ allows $G$ to shuffle things in $\C^S$ around (potentially in different ways in different regions) via its action on $S$. More precise details will be given in Section~\ref{sec:groups}.

The $SV_G$ turn out to have a number of important properties, in particular they are always simple. Finite simple groups are, quite famously, classified, but infinite simple groups remain extremely mysterious. Since any sort of classification of infinite simple groups is way beyond the scope of current mathematics, an active current direction of research is to find new examples of simple groups with interesting or unexpected properties. Just to highlight a few recent examples: Lodha found a finitely presented simple group that acts by homeomorphisms on the circle, but not by $C^1$-diffeomorphisms or piecewise linear homeomorphisms \cite{lodha19}, Hyde and Lodha found examples of finitely generated left-orderable simple groups \cite{hyde19}, and Matte Bon and Triestino found further such examples with additional interesting properties \cite{mattebon20}, e.g., they admit no infinite subgroups with property $(T)$.

Perhaps the most striking application of twisted Brin--Thompson groups is that if $G$ and $SV_G$ are finitely generated, then $G$ embeds quasi-isometrically as a subgroup of $SV_G$. This proves:

\medskip

\begin{theorem*}\cite{belkZ}
Every finitely generated group quasi-isometrically embeds as a subgroup of a finitely generated simple group.
\end{theorem*}

\medskip

This strengthens a result of Bridson \cite{bridson98}, and provides a geometric version of the classical result due to Hall that every finitely generated group embeds in a finitely generated simple group \cite{hall74}. This theorem was also proved independently, through completely different means, by Darbinyan and Steenbock \cite{darbinyan}.

\medskip

Another application of twisted Brin--Thompson groups is that they provide a family of simple groups with arbitrary finiteness length. Recall that a group is of \emph{type $F_n$} if it acts geometrically (that is, properly and cocompactly) on an $(n-1)$-connected CW-complex. In a sense, type $\F_n$ is a group theoretic analogue of the topological property of being $(n-1)$-connected. In particular, type $\F_1$ is equivalent to finite generation, and type $\F_2$ is equivalent to finite presentability. Equivalently, a group is of type $\F_n$ if and only if it admits a classifying space with finite $n$-skeleton. We also say \emph{type $\F_\infty$} to mean type $\F_n$ for all $n$, and note that every group is of type $\F_0$. The \emph{finiteness length} of a group is the largest $n\in\N\cup\{0,\infty\}$ such that the group is of type $\F_n$.

In \cite{skipper19}, Skipper, Witzel, and the author constructed a family of simple groups with arbitrary finiteness length; these were the first such examples known for finiteness lengths other than $0$, $1$, and $\infty$. These examples came from so called R\"over--Nekrashevych groups, and rely on utilizing very specific constructions of self-similar groups. Twisted Brin--Thompson groups provide another such family, this time without relying on self-similar groups, thus recovering the following:

\medskip

\begin{theorem*}\cite{skipper19}
For all $n\in\N$ there exists a simple group of type $\F_{n-1}$ but not type $\F_n$.
\end{theorem*}

\medskip

The third main application of twisted Brin--Thompson groups done in \cite{belkZ} was to produce examples of simple groups of type $\F_\infty$ containing every right-angled Artin group. A \emph{right-angled Artin group (RAAG)} is a group definable by a finite presentation in which every defining relator is a commutator of two generators. Subsequently, Salo proved in \cite{salo} that in fact every RAAG already embeds into $2V$, which is a much easier group to deal with than the twisted Brin--Thompson group examples given in \cite{belkZ}. However, the construction in \cite{belkZ} can also be used to produce simple groups with arbitrary finiteness length containing every RAAG:

\medskip

\begin{main:sep_all_raags}
For all $n\in\N$ there exists a simple group of type $\F_{n-1}$ but not type $\F_n$ into which every RAAG embeds as a subgroup.
\end{main:sep_all_raags}

\medskip

This was not explicitly stated in \cite{belkZ}, but follows easily from the results there, as we will explain in Section~\ref{sec:properties}.

\medskip

In the coming sections we recall the construction of the groups $SV_G$, discuss the embedding results from \cite{belkZ}, and the results about finiteness properties. We also prove a strengthening of a result from \cite{belkZ}, which is the main result of this note:

\medskip

\begin{main:main}
Let $G$ be a finitely presented group acting faithfully and oligomorphically on a set $S$. Suppose the stabilizer in $G$ of any finite subset of $S$ is finitely generated. Then $G$ embeds as a subgroup of a finitely presented simple group, namely the twisted Brin--Thompson group $SV_G$.
\end{main:main}

\medskip

Here an action is \emph{oligomorphic} if there are finitely many orbits of $k$-element subsets, for each $k\in\N$. To be precise, this strengthens the $n=2$ case of \cite[Theorem~D]{belkZ} by only requiring the stabilizers to be finitely generated, not necessarily finitely presented. This in particular proves that such a $G$ (and indeed any of its finitely generated subgroups) must have decidable word problem, being as it is embeddable in a finitely presented simple group, a fact also confirmed directly by James Hyde \cite{hydeMO}. We believe that trying to embed a given group into a finitely presented group admitting such an action is ``easier'' than trying to embed it directly into a finitely presented simple group, so this reduction could have applications toward resolving the Boone--Higman Conjecture for certain groups (see Conjecture~\ref{conj:bh}).

\subsection*{Acknowledgments} Thanks are due to Jim Belk for helpful discussions. The author is supported by grant \#635763 from the Simons Foundation.

\section{The construction of the groups}\label{sec:groups}

The easiest way to understand the twisted Brin--Thompson groups $SV_G$ is to focus on the case when $S$ is finite, and extrapolate to infinite $S$ by analogy. Thus, let us first focus on the case when $S$ is finite, and let us also first assume there is no twisting (so, the situation Brin first considered in \cite{brin04}).

\subsection{Finite $S$ case, no twisting}\label{ssec:finite_S_no_twisting}

Let $S=\{1,\dots,s\}$, let $\C=\{0,1\}^\N$ be the Cantor set (with the usual product topology), and let $\C^s$ be the product of $s$ many copies of $\C$. Note that an element of $\C$ is an infinite string $\kappa$ of $0$s and $1$s. Also denote by $\{0,1\}^*$ the set of all finite strings of $0$s and $1$s, including the empty string $\varnothing$. Given an $s$-tuple $\psi=(\psi(1),\dots,\psi(s))$ with $\psi(i)\in\{0,1\}^*$ (we write $\psi(i)$ instead of $\psi_i$ for the sake of future notational clarity), denote by $B(\psi)$ the \emph{cone} on $\psi$ (also called ``dyadic brick'' in \cite{belkZ}), that is, the subset of $\C^s$ given by
\[
B(\psi) \defeq \{(\psi(1)\kappa(1),\dots,\psi(s)\kappa(s))\mid \kappa(1),\dots,\kappa(s)\in \C\} \text{.}
\]
For example, $B(0,1)$ consists of all pairs $(\kappa,\kappa')$ such that $\kappa$ starts with $0$ and $\kappa'$ starts with $1$. Note that the cone $B(\psi)$ is canonically homeomorphic to $\C^s$, via the \emph{canonical homeomorphism} $h_\psi \colon \C^s \to B(\psi)$ defined via
\[
h_\psi \colon (\kappa(1),\dots,\kappa(s))\mapsto (\psi(1)\kappa(1),\dots,\psi(s)\kappa(s))\text{.}
\]

Now let us consider partitions of $\C^s$ into finitely many cones, for example $\{B(0,\varnothing),B(1,0),B(1,1)\}$ is a partition of $\C^2$ into three cones. Suppose we have a partition $\{B(\varphi_1),\dots,B(\varphi_n)\}$ of $\C^s$ into finitely many cones, and another partition $\{B(\psi_1),\dots,B(\psi_n)\}$ of $\C^s$ into the same number of cones. Then we can define a homeomorphism $\C^s\to \C^s$, by sending each $B(\varphi_i)$ to $B(\psi_i)$ via a \emph{prefix replacement}, i.e., by first applying $h_{\varphi_i}^{-1}\colon B(\varphi_i) \to \C^s$ and then $h_{\psi_i}\colon \C^s \to B(\psi_i)$. The reason for the name ``prefix replacement'' is that this sends the element $(\varphi_i(1)\kappa(1),\dots,\varphi_i(s)\kappa(s))$ to $(\psi_i(1)\kappa(1),\dots,\psi_i(s)\kappa(s))$, i.e., it replaces each prefix $\varphi_i(j)$ with $\psi_i(j)$.

\medskip

\begin{definition}[Brin--Thompson group]
The set of all homeomorphisms of this form is the \emph{Brin--Thompson group} $sV$.
\end{definition}

\medskip

It is not obvious from this definition that $sV$ is a group, but indeed any composition of homeomorphisms of this form is again of this form, essentially thanks to the fact that any two partitions into finitely many cones have a common refinement that is also a partition into finitely many cones. (This is why the set is closed under compositions, and it is obviously closed under inverses.)

\medskip

As explained in \cite[Remark~1.1]{belkZ}, it is enough to consider only certain partitions of $\C^s$ into finitely many cones, called \emph{dyadic partitions} there. A dyadic partition is one obtained by successively ``halving'' the copies of $\C^s$, that is, by iteratively replacing some cone $B(\psi(1),\dots,\psi(s))$ with the cones
\begin{align*}
&B(\psi(1),\dots,\psi(i-1),\psi(i)0,\psi(i+1),\dots,\psi(s))\\ \text{ and }&B(\psi(1),\dots,\psi(i-1),\psi(i)1,\psi(i+1),\dots,\psi(s))
\end{align*}
for some $1\le i\le s$. For example, $\{B(0,\varnothing),B(1,0),B(1,1)\}$ is a dyadic partition since it is obtained by first halving $\C^2=B(\varnothing,\varnothing)$ into $B(0,\varnothing)$ and $B(1,\varnothing)$, and then halving $B(1,\varnothing)$ into $B(1,0)$ and $B(1,1)$. It turns out every element of $sV$ is realizable with both its domain and range partitions dyadic in this way.

This halving procedure for constructing the domain and range partitions defining an element of $sV$ is convenient for pictures. Viewing $\C$ visually as an interval and $\C^s$ as an $s$-dimensional cube, an element of $sV$ is obtained by successively halving $s$-cubes in two ways, and then mapping the cones canonically. See Figure~\ref{fig:no_twist} for an example.

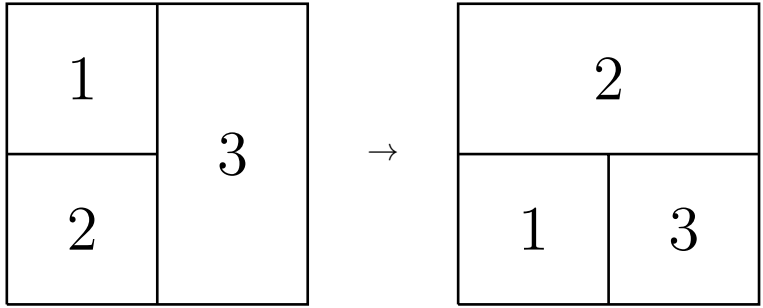
\begin{figure}[htb]\centering
 \begin{tikzpicture}[line width=1pt]
  \draw (0,0) -- (4,0) -- (4,4) -- (0,4) -- (0,0);
	\draw (2,0) -- (2,4)   (0,2) -- (2,2);
	\node[scale=2] at (1,3) {$1$};
	\node[scale=2] at (1,1) {$2$};
	\node[scale=2] at (3,2) {$3$};
	
	\node at (5,2) {$\to$};
	
	\begin{scope}[xshift=6cm]
	\draw (0,0) -- (4,0) -- (4,4) -- (0,4) -- (0,0);
	\draw (2,0) -- (2,2)   (0,2) -- (4,2);
	\node[scale=2] at (1,1) {$1$};
	\node[scale=2] at (2,3) {$2$};
	\node[scale=2] at (3,1) {$3$};
	\end{scope}
 \end{tikzpicture}
\caption{An example of an element of $2V$. We view $\C^2$ as a square, chop it into three cones in two different ways, as shown, and map them to each other via canonical homeomorphisms.}\label{fig:no_twist}
\end{figure}

\subsection{Finite $S$ case, with twisting}\label{ssec:finite_S_twisting}

Now we incorporate twisting. Let $G$ be a (finite) group acting faithfully on $S=\{1,\dots,s\}$. The action of $G$ on $S$ induces an action of $G$ on $\C^s$ by permuting the coordinates. Given $\gamma\in G$, write $\tau_\gamma\colon \C^s\to\C^s$ for the \emph{basic twist homeomorphism} induced by $\gamma$. Our notational convention is, given $\kappa\in\C^s$, the $s$-tuple $\tau_\gamma(\kappa)$ has $i$th entry $\kappa(\gamma^{-1}(i))$, so
\[
\tau_\gamma(\kappa(1),\dots,\kappa(s))\defeq \kappa(\gamma^{-1}(1)),\dots,\kappa(\gamma^{-1}(s)) \text{.}
\]
This way, the $\gamma(i)$th entry of $\tau_\gamma(\kappa)$ is the $i$th entry of $\kappa$, i.e.,
\[
\tau_\gamma(\kappa)(\gamma(i)) = \kappa(i) \text{.}
\]
The main reason for defining twists this way is so that the resulting map $\gamma\mapsto \tau_\gamma$ from $G$ to the group of twists is a homomorphism rather than an antihomomorphism, that is, $\tau_\gamma \tau_{\gamma'} = \tau_{\gamma\gamma'}$ for all $\gamma,\gamma'\in G$.

Now given two cones $B(\varphi)$ and $B(\psi)$ and an element $\gamma\in G$, we can define the \emph{twist homeomorphism} $B(\varphi)\to B(\psi)$ to be the composition $h_\psi \circ \tau_\gamma \circ h_\varphi^{-1}$. Intuitively, this acts as a prefix replacement with a twist in the middle: first the prefixes $\varphi(i)$ are deleted, then the remaining strings are permuted according to $\gamma$, and then new prefixes $\psi(i)$ are added.

We are now ready to define the twisted Brin--Thompson group $sV_G$. Given a partition of $\C^s$ into finitely many cones $B(\varphi_1),\dots,B(\varphi_n)$, another partition into the same number $n$ of cones $B(\psi_1),\dots,B(\psi_n)$, and an $n$-tuple $(\gamma_1,\dots,\gamma_n)$ of elements of $G$, we can define a homeomorphism $\C^s\to \C^s$ by sending each $B(\varphi_i)$ to $B(\psi_i)$ via the twist homeomorphism $h_{\psi_i} \circ \tau_{\gamma_i} \circ h_{\varphi_i}^{-1}$.

\medskip

\begin{definition}[Twisted Brin--Thompson group]
The set of all homeomorphisms of this form is the \emph{twisted Brin--Thompson group} $sV_G$.
\end{definition}

\medskip

Once one sees why $sV$ is a group, it is not hard to see that $sV_G$ is also a group. The key is that, if $B(\varphi)$ is a cone and $B(\varphi')$ is a ``half'' of $B(\varphi)$, so $\varphi'$ is obtained by replacing some $\varphi(i)$ with either $\varphi(i)0$ or $\varphi(i)1$, then the restriction of a twist homeomorphism on $B(\varphi)$ to $B(\varphi')$ equals a twist homeomorphism on $B(\varphi')$.

\medskip

\begin{example}
An easy but useful example to keep in mind is $G=\Z/2\Z$ acting on $S=\{1,2\}$ in the non-trivial way. See Figure~\ref{fig:twist} for an example of an element of $2V_{\Z/2\Z}$.
\end{example}

\medskip

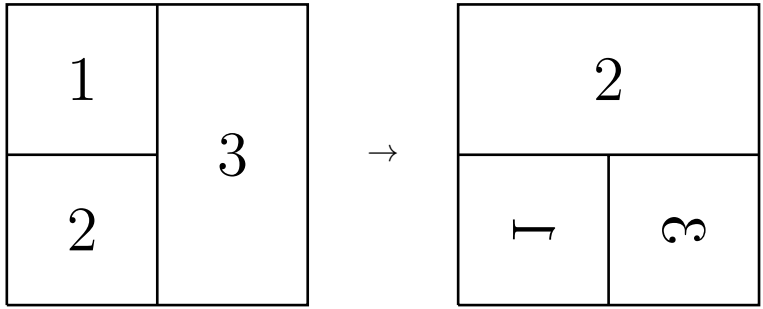
\begin{figure}[htb]\centering
 \begin{tikzpicture}[line width=1pt]
  \draw (0,0) -- (4,0) -- (4,4) -- (0,4) -- (0,0);
	\draw (2,0) -- (2,4)   (0,2) -- (2,2);
	\node[scale=2] at (1,3) {$1$};
	\node[scale=2] at (1,1) {$2$};
	\node[scale=2] at (3,2) {$3$};
	
	\node at (5,2) {$\to$};
	
	\begin{scope}[xshift=6cm]
	\draw (0,0) -- (4,0) -- (4,4) -- (0,4) -- (0,0);
	\draw (2,0) -- (2,2)   (0,2) -- (4,2);
	\node[scale=2,xscale=-1,rotate=90] at (1,1) {$1$};
	\node[scale=2] at (2,3) {$2$};
	\node[scale=2,xscale=-1,rotate=90] at (3,1) {$3$};
	\end{scope}
 \end{tikzpicture}
\caption{An example of an element of $2V_{\Z/2\Z}$, where $\Z/2\Z$ acts on $\{1,2\}$ in the non-trivial way. We view $\C^2$ as a square, chop it into three cones in two different ways, as shown, and map them to each other. The cone labeled $2$ is mapped via a canonical homeomorphism, and the cones labeled $1$ and $3$ are mapped via twist homeomorphisms (as indicated by the labels $1$ and $3$ being reflected).}\label{fig:twist}
\end{figure}

\subsection{Infinite $S$ case}\label{ssec:infinite_S}

Now we generalize to infinite $S$. Playing the role of $\C^s$ is $\C^S$, the space of functions from $S$ to $\C$, with the usual product topology. To represent a collection of prefixes, we take a function $\psi\colon S\to \{0,1\}^*$ with \emph{finite support}, that is, satisfying $\psi(s)=\varnothing$ for all but finitely many $s\in S$. Analogously to the finite $S$ case, we define the \emph{cone} $B(\psi)$ to be the basic open set
\[
\{\kappa\in \C^S\mid \psi(s) \text{ is a prefix of } \kappa(s) \text{ for each } s\in S\} \text{.}
\]
For any $\psi$ we have a \emph{canonical homeomorphism} $h_\psi \colon \C^S \to B(\psi)$ defined by
\[
h_\psi(\kappa)(s) \defeq \psi(s)\kappa(s) \text{.}
\]
As in the finite $S$ case, we can consider homeomorphisms $\C^S\to \C^S$ given by \emph{prefix replacements}, defined as follows. Take a partition of $\C^S$ into finitely many cones $B(\varphi_1),\dots,B(\varphi_n)$, take another partition into the same number of cones $B(\psi_1),\dots,B(\psi_n)$, and send each $B(\varphi_i)$ to $B(\psi_i)$ via $h_{\psi_i}\circ h_{\varphi_i}^{-1}$. These homeomorphisms comprise the \emph{Brin--Thompson group} $SV$.

Now suppose $G$ is a group acting faithfully on $S$, and we will introduce twisting. For $\gamma\in G$, let $\tau_\gamma \colon \C^S \to \C^S$ be the \emph{basic twist homeomorphism} defined by
\[
\tau_\gamma(\kappa)(\gamma(s)) \defeq \kappa(s) \text{.}
\]
Given cones $B(\varphi)$ and $B(\psi)$ we get a \emph{twist homeomorphism} $B(\varphi)\to B(\psi)$ given by $h_\psi\circ \tau_\gamma \circ h_\varphi^{-1}$. More generally, given a partition of $\C^S$ into finitely many cones $B(\varphi_1),\dots,B(\varphi_n)$, another such partition $B(\psi_1),\dots,B(\psi_n)$, and an $n$-tuple $(\gamma_1,\dots,\gamma_n)$ of elements of $G$, we can define a homeomorphism $\C^S\to\C^S$ by sending $B(\varphi_i)$ to $B(\psi_i)$ via the twist homeomorphism $h_{\psi_i}\circ \tau_{\gamma_i} \circ h_{\varphi_i}^{-1}$. These homeomorphisms comprise the \emph{twisted Brin--Thompson group} $SV_G$.

\medskip

\begin{example}\label{ex:G_act_self}
An important example to keep in mind is $S=G$ with the action given by left translation. In this case we get a twisted Brin--Thompson group denoted $GV_G$.
\end{example}

\medskip

Just to emphasize, in Example~\ref{ex:G_act_self}, $G$ can be any group whatsoever. In particular, every group embeds into a twisted Brin--Thompson group.

Perhaps the most important general property of twisted Brin--Thompson groups is the following:

\medskip

\begin{theorem*}\cite[Theorem~3.4]{belkZ}
For any $S$ and $G$, the group $SV_G$ is simple.
\end{theorem*}

\section{Embedding and finiteness properties}\label{sec:properties}

Note that $G$ embeds as a subgroup of $SV_G$, namely the subgroup of all basic twist homeomorphisms $\tau_\gamma$. Since $SV_G$ is simple, this provides a nice, natural way of embedding an arbitrary group into a simple group. Now one can begin asking whether any desirable properties of $G$ are inherited by $SV_G$.

In geometric group theory, perhaps the most desirable basic property a group could have is to be finitely generated. It is not too difficult to relate the finite generation of $G$ to that of $SV_G$:

\medskip

\begin{theorem*}\cite[Theorem~A]{belkZ}
The group $SV_G$ is finitely generated if and only if $G$ is finitely generated and the action of $G$ on $S$ has finitely many orbits.
\end{theorem*}

\medskip

For example if one uses $S=G$ as in Example~\ref{ex:G_act_self}, then $GV_G$ is finitely generated if and only if $G$ is. In particular this recovers the classical result of Hall \cite{hall74} that every finitely generated group embeds as a subgroup of a finitely generated group.

Let us now discuss the relationship between the geometry of $G$ and $SV_G$, in the case when they are both finitely generated. First we need to recall some background on basic geometric group theory.

Recall that given any finitely generated group $\Gamma$ with a fixed finite generating set $A$, we get a \emph{word metric} on $\Gamma$ by declaring that the distance between $g$ and $h$ is the length of the shortest word in the generators and their inverses representing $g^{-1}h$. It is an easy exercise to verify that this is a metric, and since $A$ is finite, $\Gamma$ is even a proper metric space. For metric spaces $X$ and $Y$ with metrics $d_X$ and $d_Y$ respectively, a function $f\colon X\to Y$ is called \emph{coarse Lipschitz} if there exist constants $C,D>0$ such that 
\[
d_Y(f(x),f(x')) \le Cd_X(x,x')+D
\]
for all $x,x'\in X$. It is called a \emph{quasi-isometric embedding} if additionally
\[
\frac{1}{C}d_X(x,x') - D \le d_Y(f(x),f(x'))
\]
for all $x,x'\in X$.

Given a finitely generated group $\Gamma$ and a finitely generated subgroup $\Gamma'$, one can ask whether the inclusion map $\Gamma'\to \Gamma$ is a quasi-isometric embedding, with respect to some word metrics. In this case we say that $\Gamma'$ \emph{quasi-isometrically embeds} as a subgroup of $\Gamma$. Since the inclusion is a homomorphism it will always be coarse Lipschitz, but if it is even a quasi-isometric embedding then, intuitively, points that are far apart in $\Gamma'$ do not become substantially closer when viewed in the ambient group $\Gamma$. Thus, if $\Gamma'$ quasi-isometrically embeds as a subgroup of $\Gamma$, then in some sense one should view $\Gamma'$ as not only an algebraic sub-object of $\Gamma$, but additionally as a geometric sub-object.

Now we can state the main application of twisted Brin--Thompson groups:

\medskip

\begin{theorem*}\cite[Theorems~B and~C]{belkZ}
Let $G$ and $SV_G$ be finitely generated. Then the inclusion map $G\to SV_G$ is a quasi-isometric embedding. Hence every finitely generated group $G$ quasi-isometrically embeds as a subgroup of some finitely generated simple group, e.g., $GV_G$.
\end{theorem*}

\medskip

Now we turn our attention to higher finiteness properties. An obvious question is, if $G$ is of type $\F_n$, what conditions on the action of $G$ on $S$ can ensure that $SV_G$ is also of type $\F_n$?

The main result in this direction in \cite{belkZ} involves oligomorphic actions:

\medskip

\begin{definition}[Oligomorphic]
An action of a group $G$ on a set $S$ is \emph{oligomorphic} if for each $k\ge 1$ there are finitely many $G$-orbits of $k$-element subsets of $S$.
\end{definition}

\medskip

\begin{theorem*}\cite[Theorem~D]{belkZ}
Suppose the action of $G$ on $S$ is oligomorphic, that $G$ is of type $\F_n$, and that for any finite subset $T\subseteq S$ the stabilizer in $G$ of $T$ is of type $\F_n$. Then $SV_G$ is of type $\F_n$.
\end{theorem*}

\medskip

Note that it is actually redundant to specify that $G$ is of type $\F_n$ here, since $G$ equals the stabilizer in $G$ of the finite subset $\emptyset$, and hence the condition on the stabilizers ensures $G$ is of type $\F_n$. For the sake of emphasizing that $G$ itself should be of type $\F_n$ though, we phrase it this way.

The $n=2$ case is worth spelling out:

\medskip

\begin{corollary}\label{cor:olig_fp}
Suppose the action of $G$ on $S$ is oligomorphic, that $G$ is finitely presented, and that for any finite subset $T\subseteq S$ the stabilizer in $G$ of $T$ is finitely presented. Then $SV_G$ is finitely presented.\qed
\end{corollary}

\medskip

In Section~\ref{sec:improved_fin_pres} we will improve this result, by only requiring the stabilizers of non-empty finite subsets to be finitely generated.

First let us discuss some of the consequences of the above theorem. Since $G$ quasi-isometrically embeds as a subgroup of $SV_G$, we see that any group $G$ acting (faithfully) oligomorphically on a set with type $\F_n$ stabilizers of finite subsets quasi-isometrically embeds as a subgroup of a simple group of type $\F_n$. This is already significant in the $n=2$ case, since any subgroup of a finitely presented simple group has decidable word problem (more on this in Section~\ref{sec:improved_fin_pres}). There are some concrete examples discussed in \cite{belkZ} that we recall here, and we refer the reader to \cite{belkZ} for more details.

\medskip

\begin{corollary*}\cite[Corollary~E]{belkZ}
Let $S=(0,1)\cap\Z[\frac{1}{2}]$ and consider the standard action of Thompson's group $F$ on $S$. This action is oligomorphic, and the stabilizer of any finite subset of $S$ is of type $\F_\infty$, so the twisted Brin--Thompson group $SV_F$ is of type $\F_\infty$.
\end{corollary*}

\medskip

\begin{corollary*}\cite[Corollary~G]{belkZ}
Let $S=\{1,\dots,n\}\times\N$ and consider the standard action of Houghton's group $H_n$ on $S$. This action is oligomorphic, and the stabilizer of any finite subset of $S$ is of type $\F_{n-1}$, so the twisted Brin--Thompson group $SV_{H_n}$ is of type $\F_{n-1}$.
\end{corollary*}

\medskip

There are also results in \cite{belkZ} showing that certain twisted Brin--Thompson groups are not of type $\F_n$. The key is that when $G$ and $SV_G$ are finitely generated, $G$ is a \emph{quasi-retract} of $SV_G$, which by a theorem of Alonso \cite{alonso94} implies that if $SV_G$ is of type $\F_n$ then so is $G$. See \cite{belkZ} for details. The important consequence of this is the following complete version of the above:

\medskip

\begin{corollary*}\cite[Corollary~G]{belkZ}
The twisted Brin--Thompson group $SV_{H_n}$ above, using Houghton's group $H_n$, is of type $\F_{n-1}$ but not type $\F_n$.
\end{corollary*}

\medskip

The $SV_{H_n}$ therefore provide an infinite family of simple groups with arbitrary finiteness length. This is the second such construction known, following the examples constructed in \cite{skipper19} using R\"over--Nekrashevych groups.

Let us remark on one final application from \cite{belkZ} of twisted Brin--Thompson groups, which has recently become a bit obsolete. By a result of Belk--Bleak--Matucci, every right-angled Artin group (RAAG) embeds into some $sV$. Thus, for any infinite $S$, any twisted Brin--Thompson group $SV_G$ contains every RAAG. In particular (see \cite[Corollary~E]{belkZ}), $SV_F$, using Thompson's group $F$ as above, provides an example of a simple group of type $\F_\infty$ containing every RAAG.

Quite recently, Salo proved that the situation can be made vastly more straightforward: $2V$ itself already contains every RAAG \cite{salo}. This is especially striking since $V=1V$ does not even contain $\Z^2*\Z$ \cite{bleak13}, and hence does not contain any RAAG whose defining graph features an edge and a vertex not adjacent to either endpoint of the edge. Thus, there is no need to use infinite $S$, or to worry about a ``twisting'' group $G$ that has a nice enough action on $S$ to get $SV_G$ to be of type $\F_\infty$. However, twisted Brin--Thompson groups do still provide the only known examples revealing the following:

\medskip

\begin{corollary}\label{cor:sep_all_raags}
For every $n\in\N$, there exists a simple group of type $\F_{n-1}$ but not $\F_n$ containing every RAAG.
\end{corollary}

\medskip

\begin{proof}
The twisted Brin--Thompson groups $SV_{H_n}$, using the Houghton groups $H_n$ as above, provide such a family, since $S=\{1,\dots,n\}\times\N$ is infinite.
\end{proof}

\medskip

Right-angled Artin groups are one of the most prominent families of groups in geometric group theory. It is natural to ask whether other prominent kinds of groups from geometric group theory can also be made to embed into simple groups with good finiteness properties.

\medskip

\begin{question}\label{quest:boone_higman}
Do each of the following groups embed (quasi-isometrically?) into a finitely presented simple group (even of type $\F_\infty$?) Braid groups, mapping class groups, $\Out(F_n)$, CAT(0) groups, hyperbolic groups.
\end{question}

\medskip

To the best of our knowledge, this question is generally open, and it would be interesting to try make progress on it using twisted Brin--Thompson groups.

\section{An improved finite presentability result}\label{sec:improved_fin_pres}

In this final section, we prove an improvement of \cite[Theorem~D]{belkZ}, which is not all that much better for $n>2$ but is, we believe, notably better for $n=2$. Let us recall the theorem and then state our improvement.

\medskip

\begin{theorem*}\cite[Theorem~D]{belkZ}
Suppose the action of $G$ on $S$ is oligomorphic, that $G$ is of type $\F_n$, and that for any finite subset $T\subseteq S$ the stabilizer in $G$ of $T$ is of type $\F_n$. Then $SV_G$ is of type $\F_n$.
\end{theorem*}

\medskip

\begin{theorem}\label{thrm:good_action_n}
Suppose the action of $G$ on $S$ is oligomorphic, that $G$ is of type $\F_n$, and that for any non-empty finite subset $T\subseteq S$ the stabilizer in $G$ of $T$ is of type $\F_{n-1}$. Then $SV_G$ is of type $\F_n$.
\end{theorem}

\medskip

In particular, when $n=2$ we get:

\medskip

\begin{theorem}\label{thrm:good_action_fp}
Suppose the action of $G$ on $S$ is oligomorphic, that $G$ is finitely presented, and that for any non-empty finite subset $T\subseteq S$ the stabilizer in $G$ of $T$ is finitely generated. Then $SV_G$ is finitely presented.\qed
\end{theorem}

\medskip

The reason we think this is notably better in the $n=2$ case is that, in general, it is much easier to prove that a group is finitely generated than to prove it is finitely presented.

\medskip

\begin{proof}[Proof of Theorem~\ref{thrm:good_action_n}]
To prove Theorem~\ref{thrm:good_action_n}, we need to delve into some of the details of the proof of \cite[Theorem~D]{belkZ} given in \cite{belkZ}. As explained in that proof, there is a contractible complex $X$ on which $SV_G$ acts, and a filtration of $X$ into $SV_G$-invariant cocompact sublevel sets $X_m$, which are increasingly highly connected as $m$ goes to $\infty$. In the proof of \cite[Theorem~D]{belkZ}, the stabilizer of any simplex was of type $\F_n$, and all of the above was enough to conclude by \cite[Theorem~6.2]{belkZ} that $SV_G$ is of type $\F_n$.

Now we have weakened the hypotheses, and so we no longer necessarily have that the stabilizer in $SV_G$ of every simplex in $X$ is of type $\F_n$. However, this was not necessary to conclude that $SV_G$ is of type $\F_n$; indeed, \cite[Theorem~6.2]{belkZ} follows from Brown's Criterion (as in \cite{brown87}), and in Brown's Criterion it is only necessary to assume that the stabilizer of every $p$-simplex is of type $\F_{n-p}$. Thus, in order to get the conclusion we want, it suffices to prove that the stabilizer in $SV_G$ of every vertex of $X$ is of type $\F_n$, and the stabilizer of every positive dimensional simplex of $X$ is of type $\F_{n-1}$. (In fact it would suffice to get the stabilizer of each $p$-simplex to be of type $\F_{n-p}$, but with twisted Brin--Thompson groups the structure of the stabilizers of higher dimensional cells is not sufficiently different than that of the stabilizers of edges for us to make any meaningful weakening of the hypotheses that leverages this.)

By \cite[Lemma~6.3]{belkZ}, the stabilizer in $SV_G$ of any vertex in $X$ is isomorphic to a wreath product of $G$ with a finite symmetric group. Since $G$ is of type $\F_n$, and since being of type $\F_n$ is preserved under taking finite direct products and finite extensions, we know that the stabilizer of any vertex is of type $\F_n$. It remains to show that positive dimensional simplices have stabilizers of type $\F_{n-1}$. For this, we can apply \cite[Corollary~6.6]{belkZ} using $n-1$ instead of $n$: since $G$ is of type $\F_n$ it is also of type $\F_{n-1}$, and our assumption is that the stabilizer in $SV_G$ of any non-empty finite subset of $S$ is of type $\F_{n-1}$, so by \cite[Corollary~6.6]{belkZ} the stabilizer of any simplex in $X$ is of type $\F_{n-1}$, as desired. (There it refers to a simplex of a complex called $|P_1|$, but $X$ is a subcomplex of this, so the same result holds.)
\end{proof}

\medskip

One concrete potential future application of Theorem~\ref{thrm:good_action_fp} (suggested by Jim Belk) is to prove that every finitely generated contracting, or finitely presented, self-similar group, and the R\"over--Nekrashevych group associated to every such self-similar group, embeds in a finitely presented simple group. Given such a self-similar group $G$, the R\"over--Nekrashevych group $V_d(G)$ is finitely presented by \cite[Theorem~5.9]{nekrashevych18} and \cite[Theorem~4.15]{skipper19}, and the commutator subgroup $[V_d(G),V_d(G)]$ is simple, but it may have infinite index. However, it seems likely that $V_d(G)$ admits an action as in Theorem~\ref{thrm:good_action_fp} and hence embeds in a finitely presented simple group. We leave this for future work.

As another remark, it would be nice to be able to only have to assume that the stabilizer in $G$ of any finite subset $T\subseteq S$ is of type $\F_{n-|T|}$, and still get to conclude that $SV_G$ is of type $\F_n$, but as remarked in the proof above there is not a convenient way to differentiate edge stabilizers from higher dimensional simplex stabilizers, so it is not clear whether something like this is possible.

\medskip

Let us conclude with a discussion of some implications and questions regarding decidability of the word problem. Recall that the Boone--Higman Theorem states that a finitely generated group has decidable word problem if and only if it embeds in a simple subgroup of a finitely presented group \cite[Theorem~I]{boone74}. The simple group can be taken to be finitely generated by a result of Thompson \cite{thompson80}. Boone and Higman also asked a question, which has since turned into a conjecture, called the Boone--Higman Conjecture:

\medskip

\begin{conjecture}[Boone--Higman]\label{conj:bh}
Every finitely presented group with decidable word problem embeds in a finitely presented simple group.
\end{conjecture}

\medskip

Since any twisted Brin--Thompson group is simple, if $SV_G$ (and hence $G$) is finitely presented, then every subgroup of $SV_G$, for instance $G$, must have decidable word problem. Since there do exist finitely presented groups with undecidable word problem, this implies that $G$ being finitely presented is not sufficient for $SV_G$ to be finitely presented. The situation in Theorem~\ref{thrm:good_action_fp} thus provides a non-obvious sufficient condition to ensure a finitely presented group (along with any of its finitely generated subgroups) has decidable word problem. As a remark, James Hyde also found a direct proof that a group admitting an action as in Theorem~\ref{thrm:good_action_fp} has decidable word problem \cite{hydeMO}. Finally, note that all of the groups in Question~\ref{quest:boone_higman} have decidable word problem, and thus conjecturally should embed into finitely presented simple groups; it would be interesting to try and resolve this by proving that they embed into groups admitting actions as in Theorem~\ref{thrm:good_action_fp}.

\bibliographystyle{alpha}

\end{document}